\documentclass[12pt,a4paper]{amsart}
\usepackage{amssymb}
\usepackage{amscd}

\theoremstyle{plain}
\newtheorem{thm}{Theorem}[section]
\newtheorem{lem}[thm]{Lemma}

\newtheorem{cor}[thm]{Corollary}

\theoremstyle{definition}

\theoremstyle{remark}

\newcommand{\R}{\mathbb{R}}

\newcommand{\N}{\mathbb{N}}

\DeclareMathOperator{\Hom}{Hom}

\DeclareMathOperator{\Ker}{Ker}

\DeclareMathOperator{\Ima}{Im}

\begin{document}
\title[Decomposition of persistence modules]{Decomposition of pointwise finite-dimensional persistence modules}

\author{William Crawley-Boevey}
\address{Department of Pure Mathematics, University of Leeds, Leeds LS2 9JT, UK}
\email{w.crawley-boevey@leeds.ac.uk}


\begin{abstract}
We show that a persistence module (for a totally ordered indexing set) consisting of finite-dimensional
vector spaces is a direct sum of interval modules.
The result extends to persistence modules with the descending chain condition on images and kernels.
\end{abstract}
\maketitle

\section{Introduction}
We fix a base field $k$ and a totally ordered indexing set $(R,<)$, and
assume throughout that $R$ has a countable subset which is dense in the order topology on $R$.
A \emph{persistence module} $V$ is a functor from $R$, considered in the natural way as a category,
to the category of vector spaces. Thus it consists of vector spaces $V_t$ for $t\in R$ and
linear maps $\rho_{ts}:V_s\to V_t$ for $s\le t$, satisfying $\rho_{ts} \rho_{sr} = \rho_{tr}$ for all $r\le s\le t$
and $\rho_{tt} = 1_{V_t}$ for all $t$.
We say that a persistence module $V$ is \emph{pointwise finite-dimensional} if all $V_t$ are finite-dimensional.
A subset $I\subseteq R$ is an \emph{interval} if it is non-empty
and $r\le s\le t$ with $r,t\in I$ implies $s\in I$.
The corresponding \emph{interval module} $V=k_I$ is given by $V_t = k$ for $t\in I$, $V_t=0$ for $t\notin I$
and $\rho_{ts}=1$ for $s,t\in I$ with $s\le t$.

Persistence modules indexed by the natural numbers were introduced in \cite{ZC} as algebraic objects
underlying persistent homology, but other indexing sets are of interest, especially the
real numbers---see \cite{CZCG,CCSGGO,CdSGO,CdSO,L}.
As an example, to analyse a set of data points in Euclidean space, fixing $n$,
one obtains a persistence module indexed by the positive real numbers 
by taking $V_t$ to be the $n$-th homology of a union of closed balls 
of radius $t$ centred on the points. (A more computational efficient construction 
is to use the Vietoris-Rips complex.)
Assuming that this persistence module is a direct sum of interval modules, it is found that the
collection of such intervals, the \emph{barcode}, is a useful invariant of the data;
see for example \cite{CZCG,Gh}.
Our aim is to give a short proof of the following result, which enables the use of barcodes
for all pointwise finite-dimensional persistence module.

\begin{thm}
\label{t:ptwise}
Any pointwise finite-dimensional persistence module is a direct sum of interval modules.
\end{thm}

The result was already known for finite indexing sets, for persistence modules of `finite type' 
indexed by $\N$ or $\R$, see \cite[\S 5.2]{CZCG}, and, using work of Webb \cite{W}, for
pointwise finite-dimensional persistence modules indexed by a locally finite subset of $\R$, 
see \cite[Theorem 1.4]{CdSGO}.

More generally, we say that $V$ has the
\emph{descending chain condition on images and kernels} provided that for all $t,s_1,s_2,\ldots \in R$
with $t\ge s_1> s_2> \dots$, the chain
\[
V_t \supseteq \Ima \rho_{ts_1} \supseteq \Ima \rho_{ts_2} \supseteq \dots
\]
stabilizes, and for all $t,r_1,r_2,\ldots \in R$
with $t\le \dots < r_2 < r_1$, the chain
\[
V_t \supseteq \Ker \rho_{r_1t} \supseteq \Ker \rho_{r_2t} \supseteq \dots
\]
stabilizes.

\begin{thm}
\label{t:dcc}
Any persistence module with the descending chain condition on images and kernels is a direct sum of interval modules.
\end{thm}

Our proof of Theorem \ref{t:dcc} uses a variant of the 
`functorial filtration' method, which we learnt from \cite{R}.
The situation here is rather simpler, however, with no `band modules', and
we haven't found it necessary to exploit functoriality.

Since the endomorphism ring of an interval module $k_I$ is $k$, a local ring,
it follows from the Krull-Remak-Schmidt-Azumaya theorem that
the multiplicity of $k_I$ as a summand of $V$ is uniquely determined. 
Our method of proof shows that this is
also the dimension of a certain vector space $V^+_I/V^-_I$.

One can study purity for persistence modules, as they form a locally
finitely presented abelian category, see for example \cite[\S 1.2, Theorem]{CBlfp}.
In this language we prove that a persistence module $V$
satisfies the descending chain condition on images and kernels
if and only if it is $\Sigma$-pure-injective, that is, any direct sum of copies of $V$ is pure-injective.

\section{Cuts}
For the rest of the paper, let $V$ be a persistence module with the 
descending chain condition on images and kernels.
A \emph{cut} for $R$ is a pair $c=(c^-,c^+)$ of subsets of $R$
such that $R = c^- \cup c^+$ and $s<t$ for all $s\in c^-$ and $t\in c^+$.
If $c$ is a cut with $t\in c^+$, we define subspaces $\Ima^-_{ct} \subseteq \Ima^+_{ct} \subseteq V_t$ by
\[
\Ima^-_{ct} = \bigcup_{\substack{s\in c^-}} \Ima \rho_{ts},
\qquad
\Ima^+_{ct} = \bigcap_{\substack{s\in c^+ \\ s\le t}} \Ima \rho_{ts},
\]
and if $c$ is a cut with $t\in c^-$, we define subspaces $\Ker^-_{ct} \subseteq \Ker^+_{ct} \subseteq V_t$ by
\[
\Ker^-_{ct} = \bigcup_{\substack{r\in c^- \\ t\le r}} \Ker \rho_{rt},
\qquad
\Ker^+_{ct} = \bigcap_{\substack{r\in c^+}} \Ker \rho_{rt}.
\]
By convention $\Ima^-_{ct} = 0$ if $c^-$ is empty and $\Ker^+_{ct} = V_t$ if $c^+$ is empty.

\begin{lem}
\label{l:Jrealized}
Let $c$ be a cut.

(a) If $t\in c^+$, then $\Ima^+_{ct} = \Ima \rho_{ts}$ for some $s\in c^+$ with $s\le t$.

(b) If $t\in c^-$ and $c^+\neq\emptyset$, then $\Ker^+_{ct} = \Ker \rho_{rt}$ for some $r\in c^+$.
\end{lem}

\begin{proof}
(a) Suppose that $\Ima^+_{ct} \neq \Ima \rho_{ts}$ for all $s\in c^+$ with $s\le t$.
Let $s_1=t$. Since $\Ima^+_{ct} \neq \Ima \rho_{ts_1}$, there is $s_2 \in c^+$ with $\Ima \rho_{ts_2}$ strictly contained in $\Ima \rho_{ts_1}$.
Similarly, since $\Ima^+_{ct} \neq \Ima \rho_{ts_2}$, there is $s_3 \in c^+$ with $\Ima \rho_{ts_3}$ strictly contained in $\Ima \rho_{ts_3}$,
and so on. Then the chain $\Ima \rho_{ts_1} \supset \Ima \rho_{ts_2} \supset \dots $ doesn't stabilize, which is a contradiction.
Part (b) is similar.
\end{proof}

\begin{lem}
\label{l:cutrho}
Let $c$ be a cut and $s\le t$.

(a) If $s,t\in c^+$, then $\rho_{ts}(\Ima^\pm_{cs}) = \Ima^\pm_{ct}$, and

(b) If $s,t\in c^-$, then $\rho_{ts}^{-1}(\Ker^\pm_{ct}) = \Ker^\pm_{cs}$, so
$\rho_{ts}(\Ker^\pm_{cs}) \subseteq \Ker^\pm_{ct}$.
\end{lem}

\begin{proof}
It is clear that $\rho_{ts}(\Ima^+_{cs}) \subseteq \Ima^+_{ct}$.
Now $\Ima^+_{cs} = \Ima \rho_{sr}$ for some $r\in c^+$ with $r\le s$.
Then
\[
\rho_{ts}(\Ima^+_{cs}) = \rho_{ts}(\Ima \rho_{sr}) = \Ima \rho_{tr} \supseteq \Ima^+_{ct}.
\]
The rest is straightforward.
\end{proof}

\section{Intervals}
If $I$ is an interval, there are uniquely determined cuts $\ell$ and $u$ with $I=\ell^+ \cap u^-$.
Explicitly,
\begin{eqnarray*}
\ell^- = \{ t : \text{$t< s$ for all $s \in I$}\},
&\ell^+ = \{ t : \text{$t \ge s$ for some $s \in I$}\},
\\
u^+ = \{ t : \text{$t > s$ for all $s \in I$}\},
&u^- = \{ t : \text{$t\le s$ for some $s \in I$}\}.
\end{eqnarray*}
For $t\in I$ we define $V^-_{It} \subseteq V^+_{It} \subseteq V_t$ by
\begin{eqnarray*}
V^-_{It} &= &(\Ima^-_{\ell\, t} \cap \Ker^+_{ut}) + (\Ima^+_{\ell\, t} \cap \Ker^-_{ut}), \text{ and}\\
V^+_{It} &= &\Ima^+_{\ell\, t} \cap \Ker^+_{ut}.
\end{eqnarray*}

\begin{lem}
\label{l:Hprop}
For $s\le t$ in $I$ we have $\rho_{ts}(V^\pm_{Is}) = V^\pm_{It}$, and the map
\[
\overline\rho_{ts}:V^+_{Is}/V^-_{Is}\to V^+_{It}/V^-_{It}
\]
is an isomorphism.
\end{lem}

\begin{proof}
This follows from Lemma~\ref{l:cutrho}. For example, if $h\in \Ima^+_{\ell\, t} \cap \Ker^-_{ut}$,
then $h = \rho_{ts}(g)$ for some $g\in \Ima^+_{\ell\, s}$. But then
\[
g\in \rho_{ts}^{-1}(h) \subseteq \rho_{ts}^{-1}(\Ker^-_{ut}) = \Ker^-_{us}
\]
so $g\in V^-_{It}$ and $h = \rho_{ts}(g) \in \rho_{ts}(V^-_{It})$.

The map $\overline\rho_{ts}$ is clearly surjective. To show it is injective, we show
that $V^+_{Is}\cap \rho_{ts}^{-1}(V^-_{It}) \subseteq V^-_{Is}$. Again this follows from Lemma~\ref{l:cutrho}.
Suppose that $g\in V^+_{Is}$ and that $h = \rho_{ts}(g)\in V^-_{It}$. Then $h = h_1+h_2$ with
$h_1 \in \Ima^-_{\ell\, t} \cap \Ker^+_{ut}$ and $h_2 \in \Ima^+_{\ell\, t} \cap \Ker^-_{ut}$.
Then $h_1 = \rho_{ts}(g_1)$ for some $g_1\in \Ima^-_{\ell\, s} \cap \Ker^+_{us}$.
Now $\rho_{ts}(g-g_1) = h_2\in \Ker^-_{ut}$, so $g-g_1\in \Ker^-_{us}$. Also $g-g_1\in \Ima^+_{\ell\, s}$, so $g\in V^-_{Is}$.
\end{proof}

\begin{lem}
\label{l:coinit}
Any interval $I$ contains a countable subset $S$ which is coinitial,
meaning that for all $t\in I$ there is $s\in S$ with $s\le t$.
\end{lem}

\begin{proof}
If $I$ has a mimimum element $m$, then $\{m\}$ is coinitial, so suppose otherwise.
By our standing assumption, $R$ has a countable subset $X$ which is 
dense in the order topology on $R$, so 
in particular $X$ meets any non-empty set of the form 
$(u,t) = \{ r\in R : u<r<t \}$ for $u,t\in R$.
Now $S = I\cap X$ is coinitial in $I$, for if
$t\in I$, then since $I$ has no mimimum element there is $r\in I$ with $r<t$,
and there is $u\in I$ with $u<r$. Then the set $(u,t)$ is non-empty,
so it contains some point $s\in X$. Then $u<s<t$, so $s\in S$.
\end{proof}

\section{Inverse limits}
Let $I$ be an interval. For $s\le t$ in $I$, $\rho_{ts}$ induces maps $V^\pm_{Is}\to V^\pm_{It}$, so one can consider the inverse limit
\[
V^\pm_I = \varprojlim_{t\in I} V^\pm_{It}.
\]
Note that to fit with the conventions of \cite[Chap.\ 0, \S 13.1]{G}, one should use the opposite ordering on $I$.
Letting $\pi_t : V^+_I \to V^+_{It}$ denote the natural map, one can identify
\[
V^-_I = \bigcap_{t\in I} \pi_t^{-1}(V^-_{It}) \subseteq V^+_I.
\]

\begin{lem}
\label{l:pitiso}
For any $t\in I$, the induced map
\[
\overline\pi_t : V^+_I/V^-_I\to V^+_{It}/V^-_{It}
\]
is an isomorphism.
\end{lem}

\begin{proof}
If $s\le t$, then $\rho_{ts}(V^-_{Is}) = V^-_{It}$ by Lemma~\ref{l:Hprop}.
It follows that the system of vector spaces $V^-_{It}$ for $t\in I$, with transition maps $\rho_{ts}$,
satisfies the Mittag-Leffler condition \cite[Chap.\ 0, (13.1.2)]{G}.
Now, thanks to Lemma~\ref{l:coinit}, the hypotheses of \cite[Chap.\ 0, Prop.\ 13.2.2]{G} hold
for the system of exact sequences
\[
0\to V^-_{It}\to V^+_{It}\to V^+_{It}/V^-_{It} \to 0,
\]
and hence the sequence
\[
0 \to V^-_I \to V^+_I \to \varprojlim_{t\in I} V^+_{It}/V^-_{It} \to 0
\]
is exact.
By Lemma~\ref{l:Hprop}, the maps $\overline\rho_{ts}$ are all isomorphisms,
so the inverse limit on the right hand side is isomorphic to $V^+_{It}/V^-_{It}$
for all $t\in I$, giving the result.
\end{proof}

\section{Submodules}
For each interval $I$, choose a vector space complement $W^0_I$ to $V^-_I$ in $V^+_I$.
For $t\in I$, the restriction of $\pi_t$ to $W^0_I$ is injective by Lemma~\ref{l:pitiso}.

\begin{lem}
The assignment
\[
W_{It} = \begin{cases} \pi_t(W^0_I) & (t\in I) \\ 0 & (t\notin I) \end{cases}
\]
defines a submodule $W_I$ of the persistence module $V$.
\end{lem}

\begin{proof}
If $s\le t$ in $I$, then $\rho_{ts}\pi_s = \pi_t$, so $\rho_{ts}(W_{Is}) = W_{It}$.
Also, if $s\le t$ with $s\in I$ and $t\notin I$ then $t\in u^+$,
so $W_{Is} \subseteq V^+_{Is} \subseteq \Ker^+_{us} \subseteq \Ker\rho_{ts}$.
\end{proof}

\begin{lem}
\label{l:Hcompl}
$V^+_{It} = W_{It}\oplus V^-_{It}$ for all $t\in I$.
\end{lem}

\begin{proof}
This follows from Lemma~\ref{l:pitiso}.
\end{proof}

\begin{lem}
\label{l:kIcopies}
$W_I$ is isomorphic to a direct sum of copies of the interval module $k_I$.
\end{lem}

\begin{proof}
Choose a basis $B$ of $W^0_I$.
For $b\in B$, the elements $b_t = \pi_t(b)$ are non-zero
and satisfy $\rho_{ts}(b_s)=b_t$ for $s\le t$,
so they span a submodule $S(b)$ of $W_I$ which is isomorphic to $k_I$.
Now $\{ b_t : b\in B \}$ is a basis of $W_{It}$, for all $t$, so $W_I = \bigoplus_{b\in B} S(b)$.
\end{proof}

\section{Sections}
It remains to prove that $V$ is the direct sum of the submodules $W_I$ as $I$ runs
through all intervals. We prove this using what we call in this paper a `section'. We work first in an arbitrary vector space.

By a \emph{section} of a vector space $U$ we mean a pair $(F^-,F^+)$ of
subspaces $F^- \subseteq F^+ \subseteq U$.
We say that a set $\{ (F^-_\lambda,F^+_\lambda) : \lambda\in \Lambda \}$ of sections of $U$
\emph{is disjoint} if for all $\lambda\neq \mu$,
either $F^+_\lambda \subseteq F^-_\mu$ or $F^+_\mu \subseteq F^-_\lambda$;
that it \emph{covers $U$} provided that for all subspaces $X\subseteq U$ with $X\neq U$
there is some $\lambda$ with
\[
X + F^-_\lambda \neq X + F^+_\lambda;
\]
and that it \emph{strongly covers $U$} provided that for all
subspaces $Y,Z\subseteq U$ with $Z \not\subseteq Y$
there is some $\lambda$ with
\[
Y + (F^-_\lambda \cap Z) \neq Y + (F^+_\lambda \cap Z).
\]

\begin{lem}
\label{l:dirsumcond}
Suppose that $\{(F^-_\lambda,F^+_\lambda) : \lambda\in \Lambda \}$ is a set of
sections which is disjoint and covers $U$.
For each $\lambda\in \Lambda$, let $W_\lambda$ be a subspace with $F^+_\lambda = W_\lambda \oplus F^-_\lambda$.
Then $U = \bigoplus_{\lambda\in \Lambda} W_\lambda$.
\end{lem}

\begin{proof}
Given a relation $c_{\lambda_1} + \dots + c_{\lambda_n} = 0$,
by disjointness we may assume that
$F^+_{\lambda_i} \subseteq F^-_{\lambda_1}$ for all $i>1$. But then
$c_{\lambda_1} = - \sum_{i>1} c_{\lambda_i} \in F^-_{\lambda_1}$, so it is zero.
Thus the sum is direct.

Now let $X = \bigoplus_{\lambda\in \Lambda} W_{\lambda}$ and suppose that $X \neq U$.
By the covering property, there is $\lambda$ with
$X + F^-_\lambda \neq X + F^+_\lambda$.
But then $X + F^+_\lambda = X + W_\lambda + F^-_\lambda \subseteq X + F^-_\lambda$, since $W_\lambda \subseteq X$,
a contradiction.
\end{proof}

\begin{lem}
\label{l:refine}
If $\{ (F^-_\lambda,F^+_\lambda) : \lambda\in \Lambda\}$ is a sets of section which is disjoint and covers $U$,
and $\{ (G^-_\sigma,G^+_\sigma) : \sigma\in \Sigma\}$ is a set of sections which is disjoint and strongly covers $U$,
then the set
\[
\{ (F^-_\lambda + G^-_\sigma\cap F^+_\lambda,F^-_\lambda + G^+_\sigma\cap F^+_\lambda) : (\lambda,\sigma)\in \Lambda \times \Sigma \}.
\]
is disjoint and covers $U$.
\end{lem}

\begin{proof}
Disjointness is straightforward.
Suppose given $X\neq U$.
Since the $(F^-_\lambda,F^+_\lambda)$ cover $U$, there is $\lambda$ with
\[
X + F^-_\lambda \neq X + F^+_\lambda.
\]
Now letting $Y = X+F^-_\lambda$ and $Z = F^+_\lambda$, we have $Z\not\subseteq Y$.
Thus since the $(G^-_\sigma,G^+_\sigma)$ strongly cover $U$, there is $\sigma$ with
\[
Y + (G^-_\sigma\cap Z) \neq Y + (G^+_\sigma\cap Z).
\]
Hence
\[
X + F^-_\lambda + (G^-_\sigma\cap F^+_\lambda) \neq X + F^-_\lambda + (G^+_\sigma\cap F^+_\lambda).
\]
\end{proof}

\section{Completion of the proof}
Recall that $V$ is a persistence module with the 
descending chain condition on images and kernels.

\begin{lem}
\label{l:jksections}
For $t\in R$, each of the sets

(a) $\{(\Ima^-_{ct},\Ima^+_{ct}):\text{$c$ a cut with $t\in c^+$}\}$, and

(b) $\{(\Ker^-_{ct},\Ker^+_{ct}):\text{$c$ a cut with $t\in c^-$}\}$

\noindent
is disjoint and strongly covers $V_t$.
\end{lem}

\begin{proof}
We prove (a); part (b) is similar.
If $c$ and $d$ are distinct cuts with $c^+$ and $d^+$ containing $t$,
then exchanging $c$ and $d$ if necessary, we may assume that $c^+\cap d^- \neq \emptyset$.
Now if $s\in c^+\cap d^-$, then $s< t$ and
\[
\Ima^+_{ct} \subseteq \Ima \rho_{ts} \subseteq \Ima^-_{dt},
\]
giving disjointness. Now suppose that $Y,Z\subseteq U$ with $Z \not\subseteq Y$.
Defining
\[
c^- = \{ s\in R : \Ima \rho_{ts} \cap Z \subseteq Y \},
\quad
c^+ = \{ s\in R : \Ima \rho_{ts} \cap Z \not\subseteq Y \},
\]
clearly $c$ is a cut, $t\in c^+$, and
\[
Y + (\Ima^-_{ct} \cap Z) = Y + (\bigcup_{\substack{s\in c^-}} \Ima \rho_{ts} \cap Z) = \bigcup_{\substack{s\in c^-}} (Y + (\Ima \rho_{ts} \cap Z)) = Y.
\]
By Lemma~\ref{l:Jrealized}, we have $\Ima^+_{ct} = \Ima \rho_{ts}$ for
some $s\in c^+$ with $s\le t$, so
\[
Y + (\Ima^+_{ct} \cap Z) = Y + (\Ima \rho_{ts} \cap Z) \neq Y,
\]
giving the strong covering property.
\end{proof}

\begin{proof}[Proof of Theorem~\ref{t:dcc}]
For $I$ an interval and $t\in I$, we consider the section $(F^-_{It},F^+_{It})$ of $V_t$ given by
\[
F^\pm_{It} = \Ima^-_{\ell\, t} + \Ker^\pm_{ut} \cap  \Ima^+_{\ell\, t},
\]
where $\ell$ and $u$ are the cuts determined by $I$.
As $I$ runs through all intervals containing $t$, the cuts $\ell$ and $u$ run through all cuts with $t\in\ell^+$ and $t\in u^-$.
Thus by Lemmas~\ref{l:refine} and \ref{l:jksections},
the set of sections $(F^-_{It},F^+_{It})$ is disjoint and covers $V_t$.

Now by Lemma~\ref{l:Hcompl} we have $V^+_{It} = W_{It}\oplus V^-_{It}$ for all $t\in I$.
It follows that $F^+_{It} = W_{It}\oplus F^-_{It}$.
Thus by Lemma~\ref{l:dirsumcond}, the space $V_t$ is the direct sum of the spaces $W_{It}$ as $I$ runs
through all intervals containing $t$.
Thus $V$ is the direct sum of the submodules $W_I$, and each of these is a direct sum of copies of $k_I$
by Lemma~\ref{l:kIcopies}.
\end{proof}

We now turn to the results mentioned at the end of the introduction.

\begin{cor}
\label{c:mult}
The multiplicity of $k_I$ as a summand of $V$ is equal to the dimension of $V^+_I/V^-_I$.
\end{cor}

\begin{proof}
It is the multiplicity of $k_I$ as a summand of $W_I$, which by Lemma~\ref{l:kIcopies} is
the dimension of $W_I^0$, so of $V_I^+/V_I^-$.
\end{proof}

Now let $V$ be an arbitrary persistence module.

\begin{cor}
\label{c:sigma}
The following statements are equivalent.

(i) $V$ satisfies the descending chain condition on images and kernels.

(ii) $V$ is $\Sigma$-pure-injective.
\end{cor}

\begin{proof}
See \cite[\S 3.5, Theorem 1]{CBlfp} for
four equivalent conditions for $\Sigma$-pure-injectivity.
If (i) holds for $V$, then it holds for any product of copies of $V$, and so by
Theorem~\ref{t:dcc}, any product of copies of $V$ is a direct sum of interval modules.
Then (ii) follows by condition (4) of the cited theorem.
Conversely, suppose that (ii) holds. For $t\in R$ we can identify
$V_t$ with $\Hom(X,V)$ where $X$ is the interval module associated to the 
interval $I=\{s : s\ge t\}$, and $X$ is finitely presented since it is isomorphic to $\Hom(t,-)$.
Now it is easy to see that the spaces $\Ima \rho_{ts}$ and $\Ker \rho_{rt}$ 
for $s\le t\le r$ are subgroups of finite definition of $V_t$, so 
condition (3) of the cited theorem implies~(i).
\end{proof}

\section*{Acknowledgements}

I would like to thank Vin de Silva for introducing me to persistence modules,
Michael Lesnick for prompting me to write this paper and invaluable comments
on a first draft, and both of them for many stimulating questions.

\end{document}